\documentclass[11pt]{article}
\usepackage[doc]{optional}
\usepackage{enumitem}
\usepackage{mathtools,esvect}
\usepackage{color}
\usepackage{float}
\usepackage{subcaption}
\usepackage{soul}
\usepackage{graphicx}
\definecolor{labelkey}{rgb}{0,0.08,0.45}
\definecolor{refkey}{rgb}{0,0.6,0.0}

\usepackage[left]{lineno}
\usepackage{blindtext}
\usepackage{hyperref}
\usepackage[hyperref=true,
            isbn=false,
            style=numeric-comp,
            maxcitenames=3,
            maxbibnames=100,
            block=none]{biblatex}
\usepackage{mathpazo}
\usepackage{amsmath}
\usepackage{amssymb}
\usepackage{theorem}
\usepackage{fancyhdr}

\usepackage[margin=1.0in]{geometry}

\newcommand{\menge}[2]{\big\{{#1}~\big |~{#2}\big\}}

\newcommand{\fenv}[1]%
{\ensuremath{\,\overrightarrow{\operatorname{env}}_{#1}}}
\newcommand{\benv}[1]%
{\ensuremath{\,\overleftarrow{\operatorname{env}}_{#1}}}

\newcommand{\scal}[2]{\left\langle{#1},{#2}  \right\rangle}

\newcommand{\RR}{\ensuremath{\mathbb R}}

\newcommand{\NN}{\ensuremath{\mathbb N}}

\newcommand{\lspan}{\ensuremath{\operatorname{span}}}
\newcommand{\aff}{\ensuremath{\operatorname{aff}}}

\newcommand{\Fix}{\ensuremath{\operatorname{Fix}}}

\newcommand{\Id}{\ensuremath{\operatorname{Id}}}

\DeclareMathOperator{\circum}{circumcenter}
\DeclareMathOperator{\dist}{dist}

{\begin{list}{}{%
\settowidth{\labelwidth}{\textrm{#1~}}%
\setlength{\leftmargin}{\labelwidth+\labelsep}}}
{\end{list}}
\newtheorem{theorem}{Theorem}[section]
\newtheorem{lemma}[theorem]{Lemma}
\newtheorem{corollary}[theorem]{Corollary}
\newtheorem{proposition}[theorem]{Proposition}
\newtheorem{definition}[theorem]{Definition}
\theoremstyle{plain}{\theorembodyfont{\rmfamily}
}
\theoremstyle{plain}{\theorembodyfont{\rmfamily}
}
\theoremstyle{plain}{\theorembodyfont{\rmfamily}
}
\theoremstyle{plain}{\theorembodyfont{\rmfamily}
\newtheorem{example}[theorem]{Example}}

\theoremstyle{plain}{\theorembodyfont{\rmfamily}
}

\def\doi{DOI}

\newcounter{count}

\allowdisplaybreaks

\begin{document}

\title{\textrm{Circumcentering the Douglas--Rachford method}}
\author{R. Behling\thanks{Department of Exact Sciences, Federal University of Santa Catarina.
Blumenau, SC -- 88040-900, Brazil. E-mail:
\texttt{\{roger.behling,l.r.santos\}@ufsc.br.}}~~~~~
J.Y.\ Bello Cruz\thanks{Department of Mathematical Sciences, Northern Illinois University. Watson Hall 366, DeKalb, IL -- 60115-6117, USA. E-mail:
\texttt{yunierbello@niu.edu.}}~~~~~L.-R.\ dos Santos\footnotemark[1] }

\date{\today}

\maketitle \thispagestyle{fancy}

\vskip 10mm

\begin{abstract} We introduce and study a geometric modification of the Douglas--Rach\-ford method called the Circumcentered--Douglas--Rachford method. This method iterates by taking the intersection of bisectors of reflection steps for solving certain classes of feasibility problems. The convergence analysis is established for best approximation problems involving two (affine) subspaces and both our theoretical and numerical results compare favorably to the original Douglas--Rachford method. Under suitable conditions, it is shown that the linear rate of convergence of the Circumcentered--Douglas--Rachford method is at least the cosine of the Friedrichs angle between the (affine) subspaces, which is known to be the sharp rate for the Douglas--Rachford method. We also present a preliminary discussion on the Circumcentered--Douglas--Rachford method applied to the many set case and to examples featuring non-affine convex sets.


 \end{abstract}


\section{Introduction}

Projection and reflection schemes are celebrated tools for finding a point in the intersection of finitely many sets~\cite{Censor:2014gb}, a basic problem in the natural sciences and engineering (see, {\em e.g.}, \cite{Bauschke:2006ej} and \cite{Combettes:1993dq}). Probably the Douglas--Rachford method (DRM) is one of the most famous and well-studied of these schemes (see, {\em e.g.},~ \cite{BCNPW14,BCNPW15, Lions:1979kk,Svaiter:2011fb,Bauschke:2016bt}). Also known as averaged alternating reflections method, it was introduced in \cite{Douglas:1956kk} and has recently gained much popularity, in part thanks to its good behavior in non-convex settings (see, {\em e.g.},~\cite{Hesse:2014gi,ABglobal,ABTcomb,ABT16,BNlocal,benoist,Borwein:2011dq,Hesse:2013cv,Phan:2016hl}).

This paper has a two-fold aim: (i) \emph{improving} the original DRM by means of a simple geometric modification, and (ii) meeting the demand of more satisfactory schemes for the many set case.  

Regarding (i),  the proposed scheme is \emph{greedy} by means of distance to the solution set, that is, our iterate is the best possible point relying on successive reflections onto two subspaces.   In particular, we get an improvement  towards the solution set with respect to DRM iteration, which  arises from the average of two successive reflections. Also, we get a convergence rate at least as good as DRM's with a comparable computational effort per iteration, however with  numerical results  fairly favorable. The aim (ii) will be  treated in the last section. 

To consider the problem, let $\langle \cdot, \cdot\rangle$ denote the scalar product in $\RR^n$, $\|\cdot\|$ is the induced norm, that is, the euclidean vector or matrix norm, and $P_S$ denotes the orthogonal projection onto a nonempty closed and convex set $S\subset\RR^n$. The reflection onto $S$ is given by means of $P_S$, namely, $R_S(x)=(2P_S-\Id)(x)$, where $\Id$ stands for the identity operator. Note that $P_S(x)$ is simply the midpoint of the segment $[x,R_S(x)]$.

Our results are established for the following fundamental feasibility problem. Given two \emph{subspaces}  $U$, $V$ of $\RR^n$ and any point $x\in\RR^n$, we are interested in solving the \emph{best approximation problem}~\cite{Deutsch:2001fl} of finding the closest point to $x$ in $U\cap V$, {\em i.e.},
\begin{equation}\label{eq:general}
\text{Find } \bar x\in U\cap V\text{ such that } \|\bar x-x\|=\min_{w\in U\cap V} \|w-x\|.
\end{equation} For {subspaces} there exists a nice characterization of the best approximation problem above as $\bar x=P_{U\cap V}(x)$ if, and only if, $x-\bar x\in (U\cap V)^\perp$, {\em i.e.},
\[\langle x-\bar x, w\rangle =0\qquad \forall \, w\in U\cap V.\]
The classical DRM iteration at a point $x\in\RR^n$ yields a new iterate $T(x)\in\RR^n$ given by the midpoint between $x$ and $R_VR_U(x)$. That is, the Douglas--Rachford (DR) operator $T:\RR^n\rightarrow\RR^n$ designed to solve \eqref{eq:general} reads as follows
\[T=T_{U,V}\coloneqq \frac{\Id+R_VR_U}{2}.\]
It is known that, under suitable assumptions, DRM generates a sequence $\{T^k(x)\}_{k\in\NN}$ converging to the solution of the best approximation problem \eqref{eq:general} (see~\cite{BCNPW14}).

Let us now introduce our scheme, called the Circumcentered--Douglas--Rachford method (C--DRM): from a point $x\in\RR^n$ the next iterate is the circumcenter of the triangle of vertexes $x$, $y\coloneqq R_U(x)$ and $z\coloneqq R_VR_U(x)$, denoted by
\begin{equation}\label{CDR-op}C_T(x)\coloneqq\circum\{x,y,z\}.\end{equation}
By circumcenter we mean that $C_T(x)$ is equidistant to $x$, $y$ and $z$ and lies on the affine space defined by these vertexes. For all $x\in\RR^{n}$, $C_T(x)$ exists, is unique  and  elementary computable. Existence and uniqueness are obvious if the cardinality of the set $\{x,y,z\}$ is either $1$ or  $2$. In fact, if it happens that $x=y=z$, we have $C_T(x)=x\in U\cap V$ already --- the converse is also true, that is, if $C_T(x)=x$, then $x=y=z$. This means that the set of fixed points $\Fix C_T$ of the (nonlinear) operator  $C_T$ is equal to $U\cap V$. If the cardinality $\{x,y,z\}$ is $2$ then $C_T(x)$ is the midpoint between the two distinct points. If $x$, $y$ and $z$  are distinct both existence and uniqueness follow from elementary geometry. Thus, one would only have to worry about having a situation in which  $x$, $y$ and $z$  are simultaneously distinct and collinear. This cannot happen since  reflections onto subspaces are norm preserving. More than that, will further see that the distance of $x$, $y$ and $z$ to $U\cap V$ is exactly the same. Thus, the equidistance we are asking for in \eqref{CDR-op} turns out to be a necessary condition for a solution of \eqref{eq:general}. 

Note that C--DRM has indeed a $2$-dimensional search flavor: we will show that $C_T(x)$ is the closest point to $U\cap V$ belonging to $\aff\{x,y,z\}$, the affine space defined by $x$, $y$ and $z$, whose dimension is equal to $2$, if $x$, $y$ and $z$, are distinct. This property, together with the fact that the DR point $T(x)\in\aff\{x,y,z\}$, is the key to proving a better performance of C--DRM over DRM. An immediate consequence of this nice interpretation is that C--DRM will solve problems in $\RR^2$ in at most two steps.  Figure~\ref{figure2} serves as an intuition guide and illustrates our idea.

\begin{figure}[h!]
	\centering 
\includegraphics[width=.37\textwidth]{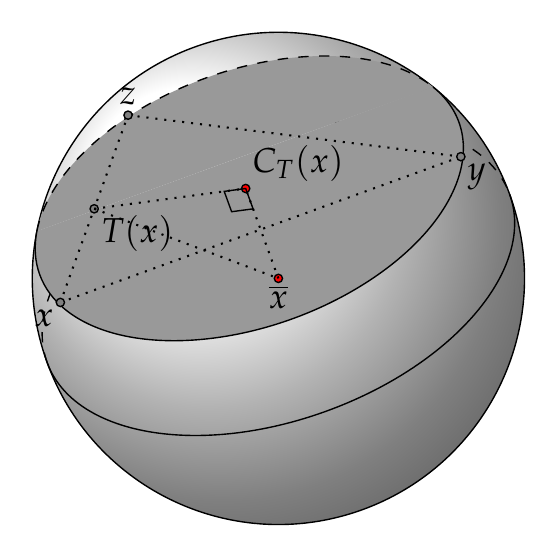}
	 \caption{Geometrical interpretation for circumcentering DRM}
	\label{figure2}
\end{figure}

This paper is organized as follows. In Section 2 we derive the convergence analysis proving that C--DRM has a rate of convergence for solving problem \eqref{eq:general} at least as good as  DRM's. Section 3 presents numerical experiments for subspaces along with some non-affine examples in $\RR^2$. Final remarks as how one can adapt C--DRM for the many set case and other comments on future work are presented in Section 4.

\section{Convergence analysis of the Circumcentered--Douglas--Rachford method}

In this section we present the theoretical advantages of using C--DRM over the classical DRM iteration for solving \eqref{eq:general}.  

In order to simplify the presentation we recall and fix some notation. For a point $x\in \RR^n$, we denote for now on  $y\coloneqq R_U(x)$ and  $z\coloneqq R_VR_U(x)$. Recall that, at the point $x\in\RR^n$, we define the C--DRM iteration as  $$C_T(x)\coloneqq  {\rm circumcenter}\{x,y,z\}$$ and the DR iteration (at the same point) is given by \[T(x)\coloneqq \frac{\Id+R_VR_U}{2}(x)=\frac{x+z}{2}.\]
In the following we present some elementary facts \cite[Theorem 5.8]{Deutsch:2001fl} needed throughout the text.

\begin{proposition}\label{Fato}
Let $S$ be a given subspace and $x \in\RR^n$ arbitrary but fixed. Then, for all $s\in S$ we have:
\begin{enumerate}
\item$ \langle x-P_S(x), s \rangle=0$;
\item $\|x-P_S(x)\|^2=\|x-s\|^2-\|s-P_S(x)\|^2$;
\item $\|x-s\|=\|R_S(x)-s\|$;
\item The projection and reflection mappings $P_S$ and $R_S$ are linear.
\end{enumerate}
\end{proposition}

Our first lemma states that the projection of $y$ and $z$ onto $U\cap V$ coincide with $P_{U\cap V}(x)$ and that the distances of $x,y$ and $z$ to $U\cap V$ are the same. In addition to that, we have that the projection of the DR point $T(x)$ onto $U\cap V$ is as well given by $P_{U\cap V}(x)$.


\begin{lemma}\label{projecion2}
Let $x\in \RR^n$. Then,  \[ P_{U\cap V}(x)=P_{U\cap V}(y)=P_{U\cap V}(z)\] and \begin{equation}\label{distanciaT}\dist (x,U\cap V)=\dist (y,U\cap V)=\dist (z,U\cap V).\end{equation}
Moreover,  $P_{U\cap V}(x)=P_{U\cap V}(T(x))$ and \[P_{U\cap V}(T^k(x))\coloneqq P_{U\cap V}(\underbrace{T(\cdots T(T}_{k}(x))\cdots )).\]
\end{lemma}
\begin{proof} We consider a bar to denote the projection onto the subspace $U \cap V$, {\em i.e.}, $\bar x\coloneqq P_{U\cap V}(x)$, $\bar y\coloneqq P_{U\cap V}(y)$, $\bar z\coloneqq P_{U\cap V}(z)$, etc.
By using Proposition \ref{Fato}(iii) for the reflection onto $U$, we get $\|x-w\|=\|y-w\|$ for all $w\in U\cap V$. In particular, $\|x-\bar y\|= \|y-\bar y\|$ and $\|y-\bar x\|= \|x-\bar x\|$ since $\bar x, \bar y \in U\cap V$. Therefore,
$$\|x-\bar x\|\le \|x-\bar y\|=  \|y-\bar y\|\le  \|y-\bar x\|= \|x-\bar x\|,$$ which implies $$\|x-\bar x\|= \|x-\bar y\|= \|y-\bar x\|=  \|y-\bar y\|,$$ 
and, of course, $\dist (x,U\cap V)=\dist (y,U\cap V)$ and $P_{U\cap V}(x)=P_{U\cap V}(y)$ follow. 

Now, the statements $\dist (y,U\cap V)=\dist (z,U\cap V)$ and $P_{U\cap V}(y)=P_{U\cap V}(z)$ can be derived by repeating the same argument with $y$ and $z$, where Proposition \ref{Fato}(iii) is then considered for the reflection onto $V$. As the projection onto subspaces is linear (Proposition \ref{Fato}(iv)) and $T(x)\coloneqq \frac{x+z}{2}$, we have $$P_{U\cap V}(T(x))=P_{U\cap V}\left(\frac{x+z}{2}\right)=\frac{P_{U\cap V}(x)}{2}+\frac{P_{U\cap V}(z)}{2}=\frac{\bar x}{2}+\frac{\bar x}{2}=\bar x.$$ The rest of the proof follows inductively. 

Indeed, we proved $P_{U\cap V}(s)=P_{U\cap V}(T(s))$ for all $s\in\RR^n$. Then, by setting $s\coloneqq T^{k-1}(x)$, we get $$P_{U\cap V}(T^{k-1}(x))=P_{U\cap V}(T(T^{k-1}(x)))=P_{U\cap V}(T^k(x)),$$ proving the lemma.  
\end{proof}

We proceed by characterizing $C_T(x)$ as the projection of any point $w\in U\cap V$ onto the affine subspace defined by $x,y$ and $z$ denoted by  $\aff\{x,y,z\}$.


\begin{lemma}\label{circum-as-proj}
Let $x\in \RR^n$ and $W^x\coloneqq \aff\{x,y,z\}$. Then, $$P_{W^x}(w)=C_T(x),$$ for all $w\in U\cap V$. In particular, $P_{W^x}(P_{U\cap V}(x))=C_T(x)$.
\end{lemma}
\begin{proof}
Let $w\in U\cap V$ be given and set $p\coloneqq P_{W^x}(w)$. Recall that $C_T(x)$ is defined by being the only point in $W^x$ equidistant to $x$, $y$ and $z$. So, to prove the lemma, we just need to show that $p$ is equidistant to $x$, $y$ and $z$. By Proposition~\ref{Fato}(ii) we have
\[
\begin{aligned}
\|x-p\|^2 &=\|x-w\|^2-\|w-p\|^2,\\
\|y-p\|^2 &=\|y-w\|^2-\|w-p\|^2,\\
\|z-p\|^2 &=\|z-w\|^2-\|w-p\|^2.
\end{aligned}
\]  
Proposition \ref{Fato}(iii) and Lemma \ref{projecion2} provided $\|x-w\|=\|y-w\|=\|z-w\|$. Hence,  the above equalities imply $\|x-p\|=\|y-p\|=\|z-p\|$, which proves the result.
\end{proof}


We have just seen that $C_T(x)$ is the closest point in $W^x=\aff\{x,y,z\}$ to $U\cap V$. In particular, the circumcenter $C_T(x)$ lies at least as close to $U\cap V$ as the DR point $T(x)$.

The next result shows that compositions of $C_T(\cdot)$ do not change the projection onto $U\cap V$, that is, for any $x\in\RR^n$ and $k\in\NN$ we have $$P_{U\cap V}(C^k_T(x)):=P_{U\cap V}(\underbrace{C_T(\cdots C_T(C_T}_{k}(x))\cdots ))=P_{U\cap V}(x).$$ This is a usual feature of algorithms designed to solve best approximation problems. 


\begin{lemma}\label{xc-proj}
Let $x\in \RR^n$ and $k\in\NN$. Then, $$P_{U\cap V}(C_T^k(x))=P_{U\cap V}(C_T(x))=P_{U\cap V}(x).$$
\end{lemma}
\begin{proof} Note that by proving the second equality, the first follows easily by induction on $k$, likewise the one in the proof of Lemma \ref{projecion2}. Therefore, let us prove that $P_{U\cap V}(C_T(x))=P_{U\cap V}(x)$. Consider again the abbreviation $\bar x\coloneqq P_{U\cap V}(x)$ and $\bar x_c\coloneqq P_{U\cap V} (C_T(x))$. From Lemma \ref{circum-as-proj} we have that $P_{W^x}(\bar x)=C_T(x)$ and also $P_{W^x}(\bar x_c)=C_T(x)$. Thus, by Pythagoras it follows that
\begin{align}
\|x-\bar x\|^2   &=\|x-C_T(x)\|^2+\|\bar x-C_T(x)\|^2, \label{(1)}\\
\|x-\bar x_c\|^2 &=\|x-C_T(x)\|^2+\|\bar x_c-C_T(x)\|^2. \label{(2)}
\end{align}
Now, using again Pythagoras, for the triangles with vertexes $x$, $\bar x$, $\bar x_c$ and $C_T(x)$, $\bar x_c$, $\bar x$, we get
\begin{align}\label{(3)}\|x-\bar x_c\|^2& =\|\bar x-\bar x_c\|^2+\|x-\bar x\|^2\\
\intertext{and}
\label{(4)}\|C_T(x)-\bar x\|^2 &=\|\bar x-\bar x_c\|^2+\|C_T(x)-\bar x_c\|^2.
\end{align}
Then, we obtain
\[
\|\bar x-\bar x_c\|^2=\|x-\bar x_c\|^2-\|x-\bar x\|^2=\|\bar x_c-C_T(x)\|^2-\|\bar x-C_T(x)\|^2=-\|\bar x-\bar x_c\|^2,
\] where the first equality follows from \eqref{(3)}, the second from subtracting \eqref{(1)} and \eqref{(2)} and the third follows from \eqref{(4)}. Thus, $\|\bar x-\bar x_c\|=0$, or equivalently, $\bar x=\bar x_c$.
\end{proof}

Note that Lemma~\ref{xc-proj} is related to  \emph{Fej\'er monotonicity} (see  \cite[Proposition 5.9 (i)]{BC2011}). 


We will now measure the improvement of $C_T(x)$ towards $U\cap V$ by means of $x$ and $T(x)$.


\begin{lemma}\label{o-melhor}
For each $x\in \RR^n$ we have 
\begin{equation}\label{(21)}\dist (C_T(x), U\cap V)^2= \dist (T(x), U\cap V)^2-\|T(x)-C_T(x)\|^2.
\end{equation}
In particular,
\begin{equation*}\label{(20)}\dist (C_T(x), U\cap V)=\|C_T(x)-P_{U\cap V}(x)\|\leq \|T(x)-P_{U\cap V}(x)\|=\dist (T(x), U\cap V).
\end{equation*}
\end{lemma}
\begin{proof}
Lemma \ref{circum-as-proj} says in particular that $P_{W^x}(P_{U\cap V}(x))=C_T(x)$, which is equivalent to saying that $\scal{s-C_T(x)}{P_{U\cap V}(x)-C_T(x)}=0$ for all $s$ on the affine subspace $W^x$. Now, taking $T(x)$ for the role of $s$ and using Pythagoras, we get \eqref{(21)}, since $P_{U\cap V}(C_T(x))=P_{U\cap V}(T(x))=P_{U\cap V}(x)$ due to Lemmas \ref{projecion2} and  \ref{xc-proj}. The rest of the result is direct consequence of Lemmas \ref{xc-proj} and \ref{projecion2}.
\end{proof}


The linear rate of convergence we are going to derive for C--DRM is given by the cosine of the Friedrichs angle between $U$ and $V$, defined below.

\begin{definition}
The \emph{cosine of the Friedrichs angle} $\theta_F$
between $U$ and $V$ is given by
\begin{equation*}
c_F(U,V) \coloneqq  \sup \menge{\scal{u}{v}}{u\in U\cap (U\cap
V)^\perp,\; v\in V\cap (U\cap V)^\perp,\;
\|u\|\leq 1,\; \|v\|\leq 1}. 
\end{equation*}
If context permits, we use just $c_F$ instead of $c_F(U,V)$.
\end{definition}

 Next we state some fundamental properties of the Friedrichs angle.

\begin{proposition}
\label{f:angle}
Let $U,V\subset\RR^n$ be subspaces, then:
\begin{enumerate}

\item 
\label{f:angleSolmon}
$0\le c_F(U,V) =c_F(V,U)= c_F(U^\perp,V^\perp)<1$.
\item
\label{f:angle3}
{$c_F=\|P_VP_U-P_{U\cap V}\|=\|P_{V^\perp}P_{U^\perp}-P_{U^\perp\cap V^\perp}\|$}. 
\end{enumerate}
\end{proposition}

\begin{proof}
(i) See \cite[Theorems 13 and 16]{Deutsch:1995ja}; (ii): See \cite[Lemma 9.5(7)]{Deutsch:2001fl}. 
\end{proof}

The following result is an elementary fact in Linear Algebra and will be used in sequel.

\begin{proposition}\label{MidpointSubspace}
Let a subspace $S\subset\RR^n$ be given. If $x,p\in \RR^n$ are such that their midpoint $$s \coloneqq \frac{x+p}{2}\in S,$$ then $\dist (x,S)=\dist (p,S)$.
\end{proposition}
\begin{proof}
Let $s\coloneqq\frac{x+p}{2}\in S$, $\hat x\coloneqq P_S(x)$ and $\hat p\coloneqq P_S(p)$. Set $\tilde p \coloneqq  \hat x+2(s-\hat x)\in S$ and note that $\tilde p$ is defined in such a way that the triangle with vertexes $x$, $\hat x$ and $s$ is congruent to the triangle formed by $p$, $\tilde p$ and $s$. In particular, $\|\tilde p-p\|=\|x-\hat x\|$. 
Therefore,
\begin{equation}\label{eq:MiddlePointSubspaces}
\|\hat p-p \|\leq\|\tilde p-p \|=\|x-\hat x\|.
\end{equation}
An analogous construction can be considered for the triangle with vertexes $p$, $\hat p$ and $s$ and the one with vertexes $x$, $\tilde x$ and $s$, where $\tilde x \coloneqq  \hat p+2(s-\hat p)\in S$, yielding $\|\hat x-x \|\leq\|\tilde x-x \|=\|p-\hat p\|$. This, combined with \eqref{eq:MiddlePointSubspaces}, proves the proposition.  
\end{proof}

The next lemma organizes and summarizes known properties of sequences generated by DRM, some of which will be important in the proof of our main theorem. It is worth noting that items (i) and (vi) are novel to our knowledge.

\begin{lemma}\label{AnyInitialPointDR}
Let $x\in \RR^n$ be given. Then, the following assertions for DRM hold:
\begin{enumerate}
	\item   $\dist (T^k(x), U+V)=\dist (x, U+V)$ for all $k\in \NN$, where $U+V =\lspan(U\cup V)$;
	\item The set $\Fix T\coloneqq \menge{x\in\RR^n}{T(x)=x}$ is given by $\Fix T =  (U\cap V)\oplus(U^\perp\cap V^\perp)$;
	\item The DRM sequence $\{T^k(x)\}_{k\in\NN}$ converges to $P_{\Fix T}(x)$ and for all $k\in \NN$, $$\|T^k(x)-P_{\Fix T}(x)\|\leq c^k_F\|x-P_{\Fix T}(x)\|;$$
	\item For all $k\in \NN$ we have $P_{U\cap V}(T^k(x))=P_{U\cap V}(x)$ and $P_{\Fix T}(T^k(x))=P_{\Fix T}(x)$;
	\item $P_{U\cap V}(x)=P_{\Fix T}(x)$ if, and only if, $x\in  U+V$;
	\item The DRM sequence $\{T^k(x)\}_{k\in\NN}$ converges to $P_{U\cap V}(x)$ if, and only if, \[x\in  U+V;\]

	\item The shadow sequences \[\{P_U(T^k(x))\}_{k\in\NN}\text{ and }\{P_V(T^k(x))\}_{k\in\NN}\] converge to $P_{U\cap V}(x)$.

\end{enumerate}
\end{lemma}
\begin{proof}
For the sake of notation set $S\coloneqq U+V$ and remind that $y\coloneqq R_U(x)$ and $z\coloneqq R_V(y)$. We have $\frac{1}{2}(x+y)=P_U(x)\in S$ and $\frac{1}{2}(y+z)=P_V(y)\in S$. Employing Proposition \ref{MidpointSubspace} yields $\dist (x,S)=\dist (y,S)=\dist (z,S)$. 
Also, $\frac{1}{2}(T(x)+y)= \frac{1}{2}(\frac{x+z}{2}+y) = \frac{1}{2}(P_U(x)+P_V(y))\in S$. Using again Proposition \ref{MidpointSubspace} we conclude that $\dist (T(x),S)=\dist (y,S)$. Hence, $\dist (T(x),S)=\dist (x,S)$ for all $x\in\RR^n$. A simple induction procedure gives us $\dist (T^k(x), U+V)=\dist (x, U+V)$ for all $k\in \NN$, proving (i).

The proofs of items (ii), (iii), (iv) and (vii)  are  in \cite{BCNPW14}. 

It is straightforward to check that $S^\perp=U^\perp\cap V^\perp$. Therefore, $\Fix T$ specialized to $S$ reduces to $U\cap V$ and (v) follows. (vi) is a combination of (ii) and (v).
\end{proof}

We are now in the position to present our main convergence result, which states that the best approximation problem \eqref{eq:general} can be solved for any point $x\in\RR^n$ with usage of C--DRM. 

\begin{theorem}\label{TeoremaSpan}
Let $x\in\RR^n$ be given. Then, the three C--DRM sequences \[\{C_T^k(P_U(x))\}_{k\in\NN}, \{C_T^k(P_V(x))\}_{k\in\NN} \text{ and } \{C_T^k(P_{ U+V}(x))\}_{k\in\NN}\] converge linearly to $P_{U\cap V}(x)$. Moreover, their rate of convergence is at least the cosine of the Friedrichs angle $c_F\in[0,1)$.
\end{theorem} 
\begin{proof}
Obviously, $u, v, s\in  U+V$, with $u \coloneqq P_U(x)$, $v \coloneqq P_V(x)$ and $s\coloneqq P_{ U+V}(x)$. Let us first prove that $\bar u=\bar v=\bar x \coloneqq P_{U\cap V}(x)$, where $\bar u \coloneqq P_{U\cap V}(u)$, $\bar v \coloneqq P_{U\cap V}(v)$ and $\bar s \coloneqq P_{U\cap V}(s)$. The definition of $\bar u,\bar x$ allow us to employ Pythagoras and conclude that
$$
\|u-x\|^2+\|u-\bar x\|^2=\|x-\bar x\|^2 \leq \|x-\bar u\|^2=\|u-x\|^2+\|u-\bar u\|^2,
$$
which provides $\|u-\bar x\|=\|u-\bar u\|$. Thus $\bar u=\bar x$. We get $\bar v=\bar s=\bar x$ analogously, and by item (v) of Lemma \ref{AnyInitialPointDR} we further have $\bar u=P_{\Fix T}(u)=\bar v=P_{\Fix T}(v)=\bar s=P_{\Fix T}(s)=\bar x$.
Hence, it holds that
\begin{align*}
\|C_T(u)-\bar x\|=\|C_T(u)-\bar u\| & \leq \|T(u)-\bar u\|=\|T(u)-P_{\Fix T}(u)\| \\
                                    & \leq c_F \|u-P_{\Fix T}(u)\|=c_F \|u-\bar x\|,
\end{align*}
where the first inequality is by Lemma~\ref{o-melhor} and the second one by item (iii) of Lemma~\ref{AnyInitialPointDR}.
Recall that Lemma \ref{xc-proj} stated that $P_{U\cap V}(C_T^k(P_U(x)))=P_{U\cap V}(P_U(x))$ for all $k\in\NN$. So, $P_{U\cap V}(C_T^k(u))=P_{U\cap V}(u)=\bar x$ for all $k\in \NN$. Consider now the induction hypothesis $\|C_T^{k-1}(u)-\bar x\|\leq c_F^{k-1}\|u-\bar x\|$ for some $k-1$ --- the case $k-1=1$ was seen above --- and note that
\begin{align*}
\|C_T^k(u)-\bar x\| &=\|C_T(C_T^{k-1}(u))-P_{U\cap V}(C_T^{k-1}(u))\| \\
                    & \leq  \|T(C_T^{k-1}(u))-P_{U\cap V}(C_T^{k-1}(u))\| 
                    \\
                    & 
                    \leq c_F \|C_T^{k-1}(u)-P_{U\cap V}(C_T^{k-1}(u))\| \\ 
                    & = c_F \|C_T^{k-1}(u) -\bar x\| 
                    \leq c_Fc_F^{k-1}\|u-\bar x\|=c_F^k\|u-\bar x\|,
\end{align*}
where the first inequality is due to Lemma~\ref{o-melhor}. The second inequality above follows from Lemma \ref{AnyInitialPointDR} (iii) and (v), since $u\in U+V$, and the third is by the induction hypothesis. This proves the theorem for the sequence $\{C_T^k(P_U(x))\}$. The proof lines for the convergence of the sequences $\{C_T^k(P_V(x))\}$ and $\{C_T^k(P_{ U+V}(x))\}$ with the linear rate $c_F$ are analogous. 
\end{proof}

An immediate consequence is stated below.

\begin{corollary}\label{CorolSpan}
Let $x\in U+V$ be given. Then, the C--DRM sequence $\{C_T^k(x)\}_{k\in\NN}$ converges linearly to $P_{U\cap V}(x)$. Moreover, the rate of convergence is at least the cosine of the Friedrichs angle $c_F\in[0,1)$.
\end{corollary}

Although we believe that Theorem \ref{TeoremaSpan} holds for $\{C_T^k(x)\}$ itself, it is worth mentioning that considering the initial feasibility step $P_U(x)$ or $P_V(x)$ is totally reasonable, since we are assuming that the projection/reflection operators $P_U$, $P_V$, $R_U$ and $R_V$ are available. In addition to that, these feasibility steps keep the whole C--DRM sequence in $ U+V$, which can be very desirable. In this sense, let us look at two distinct lines $U,V\subset\RR^3$ intersecting at the origin and assume that their Friedrichs angle is not ninety degrees, {\em i.e.}, $c_F\in(0,1)$. C--DRM finds the origin after one or two steps when starting in $ U+V$, since $U+V$ is the plane containing the two lines. If the initial point is not in $ U+V$ and no feasibility step is taken, C--DRM may generate an infinite sequence. Therefore, running C--DRM in $ U+V$, a potentially smaller set than $\RR^n$, sounds attractive from the numerical point of view.

The  feasibility procedure employed in Theorem~\ref{TeoremaSpan} provides convergence to best approximation solutions for the DRM without the need of considering shadow sequences. This is formally presented in the following. 

\begin{corollary}\label{CorolDRM-feasible}
Let $x\in \RR^{n}$ be given. Then, the three DRM sequences 
\[\{T^k(P_U(x))\}_{k\in\NN},\; \{T^k(P_V(x))\}_{k\in\NN} \; \text{ and }\;  \{T^k(P_{ U+V}(x))\}_{k\in\NN}
\] converge linearly to $P_{U\cap V}(x)$. Moreover, their rate of convergence is  given by the cosine of the Friedrichs angle $c_F\in[0,1)$.
\end{corollary}

\begin{proof} This result follows from the fact that  $P_{ U\cap V}(P_U(x))=P_{ U\cap V}(P_V(x))=P_{ U\cap V}(P_{U+V}(x)) = P_{ U\cap V}(x) $, established within the proof of Theorem~\ref{TeoremaSpan}, combined with Lemma~\ref{AnyInitialPointDR} (vi). 
\end{proof}

It is known that $c_F$ is the sharp rate for DRM \cite{BCNPW14}. This is not clear for C--DRM though and left as an open question in this work. One way of addressing this issue would be looking at the magnitude of improvement of C--DRM over DRM given by \eqref{(21)} in Lemma \ref{o-melhor}.

We finish this section by showing that Theorem~\ref{TeoremaSpan} and Corollaries~\ref{CorolSpan} and~\ref{CorolDRM-feasible} are applicable to affine subspaces with nonempty intersection, where the concept of the cosine of the Friedrichs angle is suitably extended.

\begin{corollary}\label{corol:affine-extension}
Let $A$ and $B$ be affine subspaces of $\RR^n$ with nonempty intersection and $p\in A\cap B$   arbitrary but fixed. Then, Theorem~\ref{TeoremaSpan} and Corollaries~\ref{CorolSpan} and~\ref{CorolDRM-feasible}  hold for $A$ and $B$, with the rate $c_F$ being the cosine of the Friedrichs angle between the subspaces $U_A\coloneqq A-\{p\}$ and $V_B\coloneqq B-\{p\}$.
\end{corollary}

\begin{proof}
Since $p\in A\cap B$, the translations $A-\{p\}$ and $B-\{p\}$ provide nonempty subspaces $U_A$ and $V_B$. Now, the elementary translation properties of reflections $R_{U_A}(x)+p=R_A(x+p)$ and $R_{V_B}(x)+p=R_B(x+p)$ give us the translation formulas for the correspondent Douglas-Rachford and circumcentering  operators:
\[
T_{A,B}(x+p)\coloneqq T_{U_A,V_B}(x)+p
\]
and
\[
C_{T_{A,B}}(x+p)\coloneqq C_{T_{U_A,V_B}}(x)+p,
\]
for all $x\in\RR^n$. This suffices to prove the corollary when setting $U\coloneqq U_A$ and $V\coloneqq V_B$ in Theorem~\ref{TeoremaSpan} and Corollaries~\ref{CorolSpan} and~\ref{CorolDRM-feasible}, because the above formulas imply that $ T_{A,B}^k(x+p)= T_{U_A,V_B}^k(x)+p$ and $C_{T_{A,B}}^k(x+p)= C_{T_{U_A,V_B}}^k(x)+p$,  for all $k\in \NN$.

\end{proof}

Simple manipulations let us conclude that $c_F= c_F(U_A,V_B)$, with $U_A$ and $V_B$ as above, does not depend on the particular choice of the  point $p$ in $A\cap B$. Therefore, $c_F$ can be referred to as the cosine of the Friedrichs angle between the affine subspaces $A$ and $B$ with no ambiguity.

We finish this section remarking that the computation of circumcenters is possible in arbitrary inner product spaces --- see~\eqref{eq:xc-linsys}. However, projecting/reflecting onto subspaces depends on their closedness. Now, in Hilbert spaces, it is well known that having  $c_F$ strictly smaller than $1$ is equivalent to asking $U+V$ to be closed~\cite[Theorem 13]{Deutsch:1995ja}. This is an   assumption that maintains  the linear convergence in Hilbert spaces for several projection/reflection methods and that would also enable us to extend our main results concerning C--DRM  to an infinite dimensional setting.

The following section is on numerical experiments and begins by showing how one can compute $C_T(x)$ by means of elementary and cheap Linear Algebra operations.


\section{Numerical experiments}\label{NI}

In this section, we make use of numerical experiments to show, as a proof of concept, that the good theoretical proprieties of $C_T(x)$ are also working in practical problems.

First, for a given point $x\in \RR^{n}$, we establish a procedure to find $C_T(x)$. We then discuss the stopping criteria used in our experiments and show the performance of C--DRM compared to DRM applied to problem \eqref{eq:general}. We also apply C--DRM and DRM to non-affine samples, which are not treated theoretically in the previous section. The experiments with these problems indicate as well a good behavior of C--DRM over DRM.    All experiments were performed using \texttt{julia}~\cite{Bezanson:2017gd} programming language and the pictures were generated using PGFPlots~\cite{Feuersanger:2016ut}.

\subsection{How to compute the circumcenter in $\RR^n$}

In order to compute $C_T(x)$, recall that $y \coloneqq  R_{U}(x)$ and $z \coloneqq  R_{V}(y) = R_{V}R_{U}(x)$. We have already mentioned that $C_T(x) = x$ if, and only if, the  cardinality of  $\aff\{x,y,z\}$ is 1. If this cardinality is 2, $C_T(x)$ is the midpoint between the two distinct points. 

Therefore, we will focus on the computation of $C_T(x)$ for the case where $x$, $y$ and $z$ are distinct.  According to Lemma~\ref{circum-as-proj},  $C_T(x)$ is still well defined, and it can be seen as the circumcenter of the triangle $\Delta xyz$ formed by the points $x$, $y$ and $z$, as illustrated in Figure~\ref{figure1}. Also, $T(x)$ lies in $\aff\{x,y,z\}$. 
\begin{figure}[h!]
	\centering
\includegraphics[width=.41\textwidth]{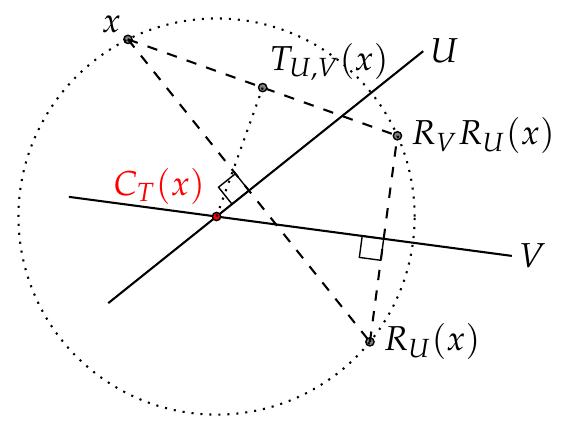}
	 \caption{Circumcenter on the affine subspace $\aff\{x,y,z\}$.}
\label{figure1}
\end{figure}

Let $x$ be \emph{the anchor} point of our framework and define $s_{U}\coloneqq y-x$ and $s_{V}\coloneqq z-x$, the vectors  pointing from $x$ to $y$ and from $x$ to $z$, respectively. Note that $\aff\{x,y,z\} = x + \lspan\{s_{U},s_{V}\}$, and since $C_T(x)\in \aff\{x,y,z\}$, the dimension of the ambient space, namely $n$, is irrelevant to the geometry. The problem is \emph{intrinsically} a $2$-dimensional one, regardless of what $n$ is.

We want to find the vector  $s\in \aff\{x,y,z\}$, whose projection onto each vector   $s_{U}$ and $s_{V}$   has its endpoint at the midpoint of the line segment from $x$ to $y$ and $x$ to $z$, that is, 
\[
P_{\lspan\{s_U\}}(s) = \frac{1}{2}s_U \quad \text{ and } \quad P_{\lspan\{s_V\}}(s) = \frac{1}{2}s_V.
\]   
This requirement yields 
\begin{equation*}
	\begin{cases}
	\left\langle s_U,s\right\rangle = \frac{1}{2} \|s_U\|^2,\\
	
	\left\langle s_V,s\right\rangle = \frac{1}{2} \|s_V\|^2.
		\end{cases}
\end{equation*}
By writing $s=\alpha s_U + \beta s_V$ we get the $2\times 2$ linear system with unique solution $(\alpha,\beta)$ 
\begin{equation}
  \label{eq:xc-linsys}
  \begin{cases}
  \alpha\| s_U\|^2 +   \beta\left\langle s_U,s_V\right\rangle  = \frac{1}{2} \|s_U\|^2,\\
   \alpha\left\langle s_U,s_V\right\rangle  + \beta\| s_V \|^2= \frac{1}{2} \|s_V\|^2.
    \end{cases}
\end{equation}
Hence,   
\[C_T(x) = x + \alpha s_U + \beta s_V.\] 
Note that \eqref{eq:xc-linsys} enables us to calculate $C_T(x)$ in arbitrary inner product spaces.

\subsection{The case of two subspaces}

In our experiments we randomly generate $100$ instances with subspaces $U$ and $V$ in $\RR^{200}$ such that $U\cap V\neq \{0\}$. Each instance is run for $20$ initial random points. Based on  Theorem~\ref{TeoremaSpan} and Corollary~\ref{CorolDRM-feasible}, we monitor the  C--DRM sequence $\{C_T^k(P_V(x))\}$ and the DRM sequence $\{T^k(P_V(x))\}$. Note that the DRM sequence will always converge to $U\cap V$, since $P_V(x)\in  U+V$. For the C--DRM sequence, such hypothesis does not seem to be necessary, however for the sake of fairness, we choose to monitor this sequence in the same way. We incorporate the method of alternating projections (MAP) (see, \emph{e.g.}, \cite{Deutsch:1985gz}) in our experiments as well. MAP generates a sequence $\{(P_VP_U)^k(x)\}$ that lies automatically in $ U+V$ for all $k\geq 1$. 

Let $\{s^k\}$ be any of the three sequences that we monitor. We considered two tolerances
\[
\varepsilon_1 \coloneqq 10^{-3} \text { and } \varepsilon_2 \coloneqq 10^{-6},
\]
and employed two stopping criteria: 
\begin{itemize}
	\item the \emph{true error}, given by 
$
\|s^k - \bar x\|< \varepsilon_{i};
$ 
and
\item the \emph{gap distance}, given by
$
\|P_U(s^k) - P_V(s^k)\|< \varepsilon_{i},
$
for $i=1,2$.
\end{itemize}
We performed tests with two tolerances in order to challenge C--DRM by asking for more precision. 
Also, one can view the \emph{true error} as the best way to assure one is sufficiently close to $U\cap V$, and in our case $\bar x\coloneqq P_{U\cap V}(x)$ is easily available. However, this is not the case in applications, thus, we also utilized the \emph{gap distance}, which we consider a reasonable measure of infeasibility.

In order to represent the results of our numerical experiments, we used the Performance Profiles~\cite{Dolan:2002du}, a analytic tool that allows one to compare several different algorithms on a set of problems with respect to a performance measure or \emph{cost}, which in our case is the number of iterations, providing the visualization and interpretation of the results of benchmark experiments.  The rationale of choosing number of iterations  as performance measure here is that in each of the sequences that are monitored, the majority of the computational cost involved is equivalent ---  two orthogonal projections onto the subspaces $U$ and $V$ per iteration have to be computed for each method. The graphics  were generated using \texttt{perprof-py}~\cite{Siqueira:2016ej}.

\begin{figure}[!ht]
\begin{subfigure}[t]{0.490\textwidth}
\centering
\includegraphics[width=\textwidth]{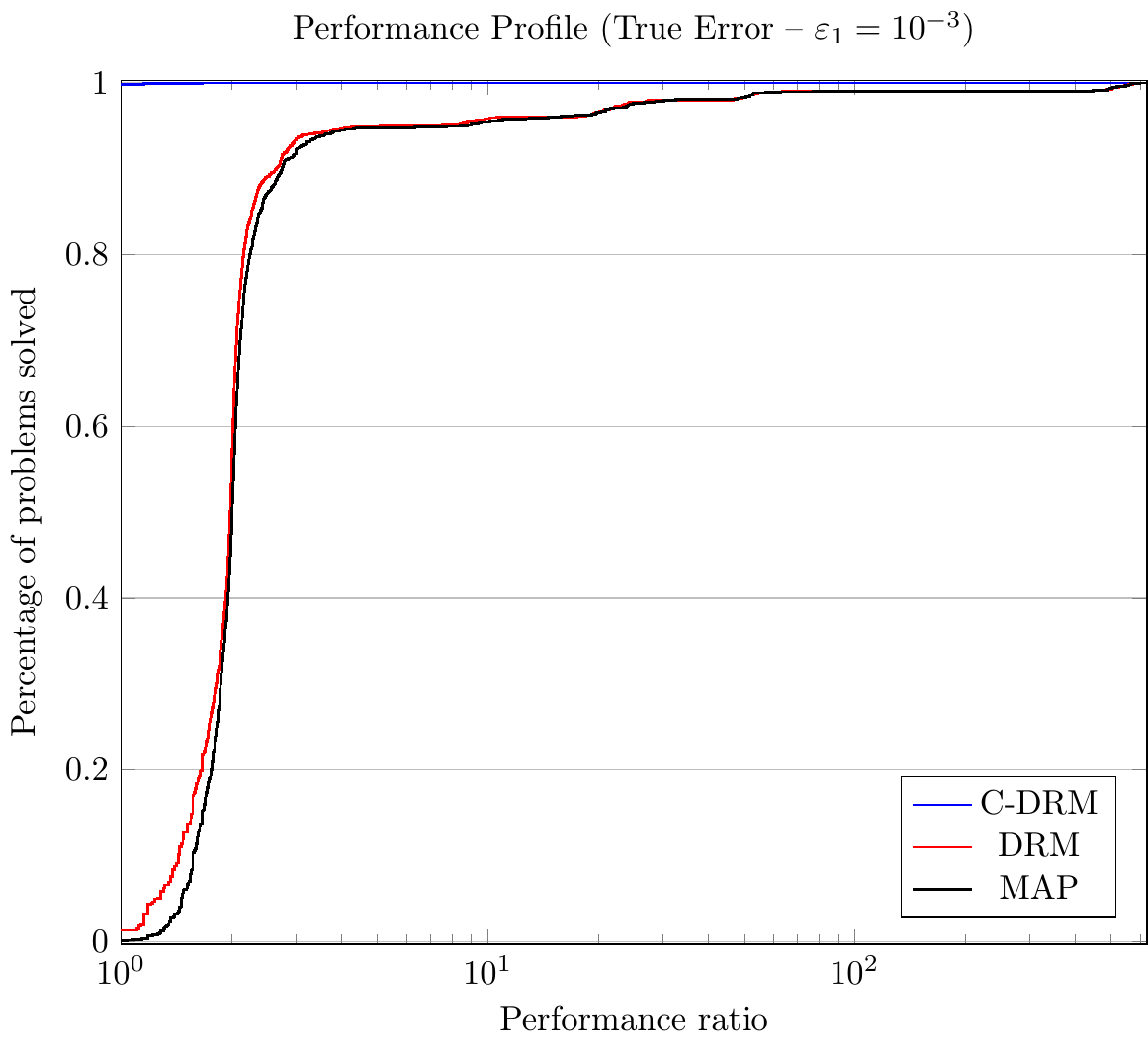}
	 \caption{\label{fig:pp-twosubspaces1}}
\end{subfigure}
  \, 
\begin{subfigure}[t]{0.490\textwidth}
\centering
\includegraphics[width=\textwidth]{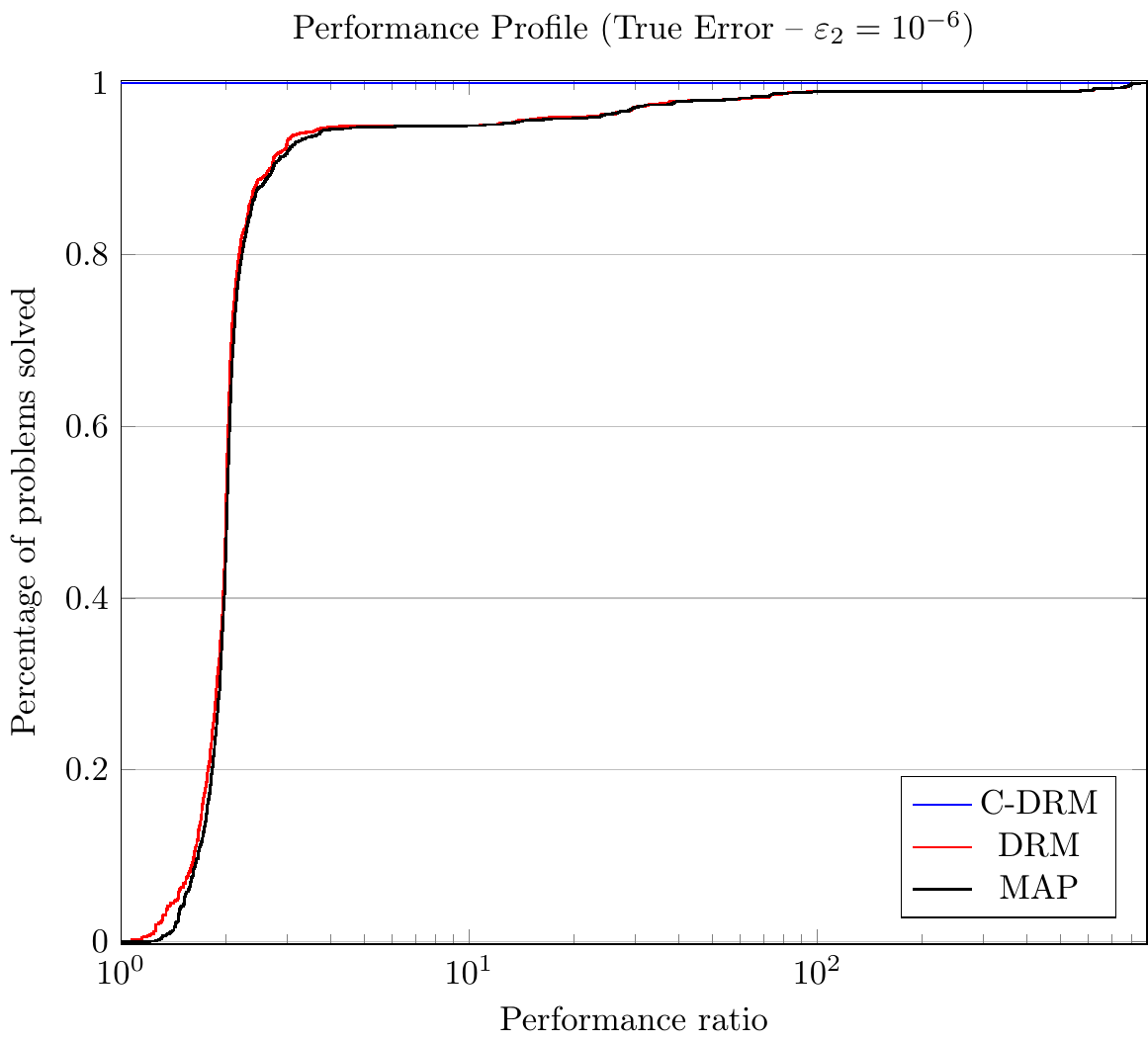}
	 \caption{\label{fig:pp-twosubspaces2}}
	
    \end{subfigure}
\\
    \begin{subfigure}[t]{0.490\textwidth}
    \centering
    \includegraphics[width=\textwidth]{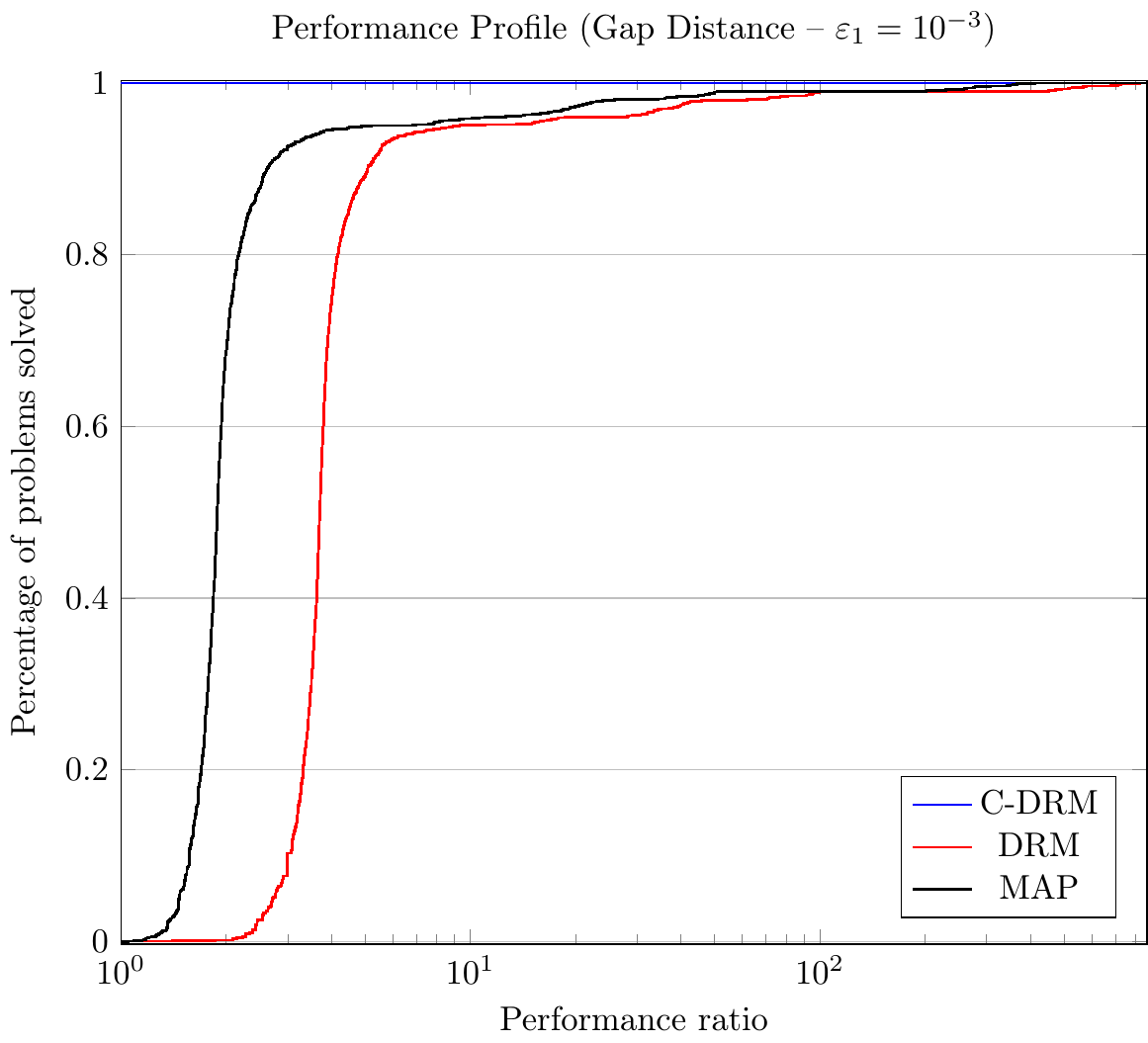}
    	 \caption{\label{fig:pp-twosubspaces3}}
    \end{subfigure}
      \, 
    \begin{subfigure}[t]{0.490\textwidth}
    \centering
    \includegraphics[width=\textwidth]{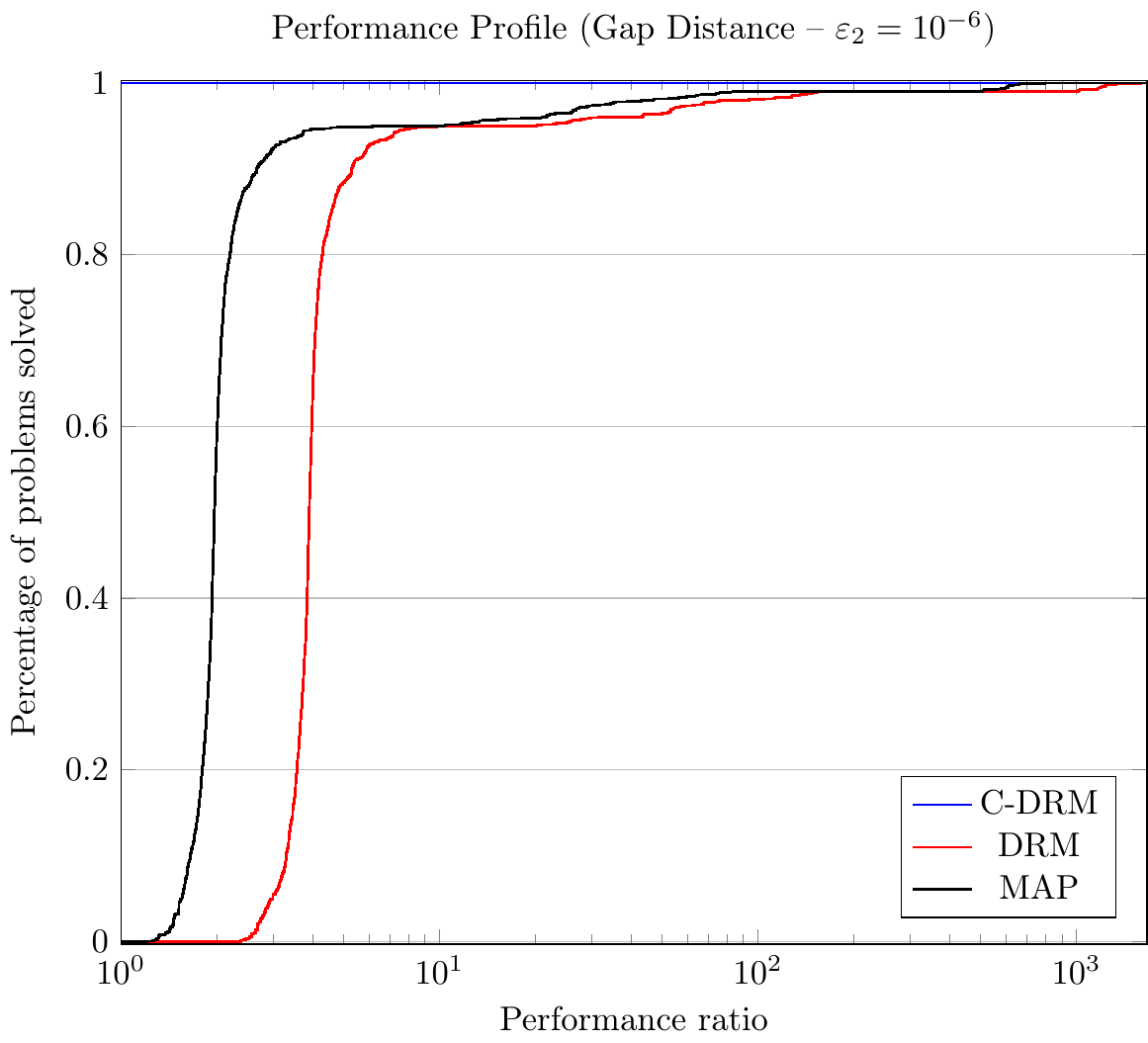}
    	 \caption{\label{fig:pp-twosubspaces4}}
    \end{subfigure}

\caption{Experiments with two subspaces.\label{fig:pp-twosubspaces}}
\end{figure}

The numerical experiments shown in Figure~\ref{fig:pp-twosubspaces}  confirm the theoretical results obtained in this paper, since C--DRM has a much better performance than DRM. For $\varepsilon_2 = 10^{-6}$, C--DRM solves all instances and choices of  initial points in less iterations than DRM, regardless the stopping criteria (See Figures~\ref{fig:pp-twosubspaces2} and \ref{fig:pp-twosubspaces4}). For  $\varepsilon_1 = 10^{-3}$, using the true error criteria, according to Figure~\ref{fig:pp-twosubspaces1} a tiny part of the instances were solved faster by DRM, however this behavior was not reproduced with the gap distance criteria (see Figure~\ref{fig:pp-twosubspaces4}). MAP has $c_F^2$  as asymptotic linear rate (see \cite{Kayalar:1988jk}) and it was beaten by C--DRM in all our instances. This gives rise to the interesting question of whether the rate $c_F^2$ can be theoretically achieved by C--DRM. 
\begin{figure}[!ht]
\begin{subfigure}[t]{0.40\textwidth}
\centering
\includegraphics[height=.25\textheight]{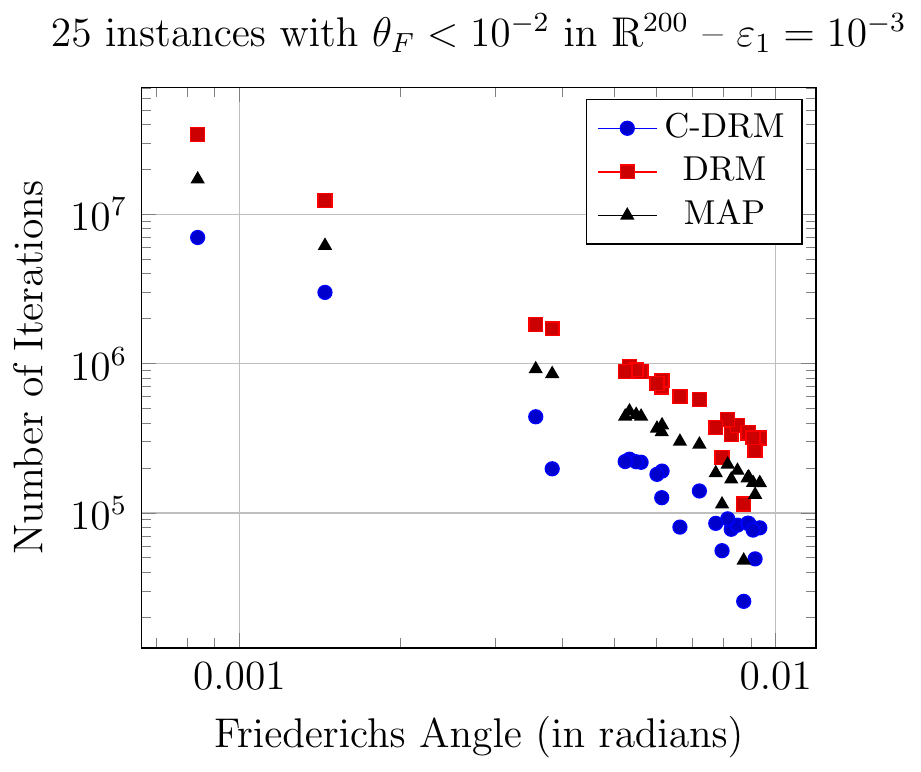}
	 \caption{\label{fig:HighcF-twosubspaces1}}
\end{subfigure}
  \hspace*{.5in}
\begin{subfigure}[t]{0.40\textwidth}
\centering
\includegraphics[height=.25\textheight]{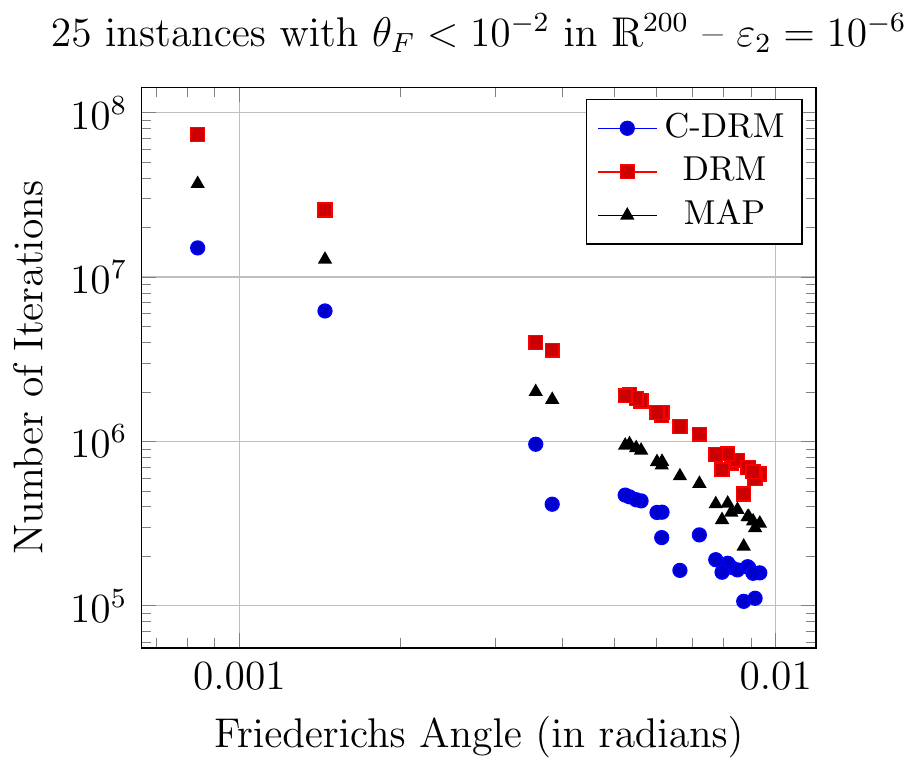}
	 \caption{\label{fig:HighcF-twosubspaces2}}
	
    \end{subfigure}

\caption{Experiments with two subspaces having a small Friederichs angle. \label{fig:HighcF-twosubspaces}}
\end{figure}

Figure~\ref{fig:HighcF-twosubspaces} represents experiments in which the Friederichs angle between the two subspaces is smaller than $10^{-2}$ and the \emph{true error} criterion is used. In this case, MAP and DRM are known to converge slowly. C--DRM, however, handled the small values of the Friederichs angle substantially better.

\subsection{Some non-affine examples}

Next we present two simple classes of examples revealing the potential of the proposed modification when it is applied to the convex and the non-convex case. Here we are using the \emph{gap distance} with $\varepsilon_2 = 10^{-6}$. 

\begin{example}[Two balls in $\RR^2$] 
 We present three figures that depict numerical experiments showing the behavior of C--DRM over DRM concerning the problem of finding a point in the intersection of two convex balls in $\RR^2$.

 Note that Lemma \ref{o-melhor} shows that $C_T(x)$ is closer to  $U\cap V$ than $T(x)$ for problem \eqref{eq:general}. Unfortunately, this is not true in general for convex sets.  Figure~\ref{fig:twoballs1}, illustrates this for two balls.  Here, $x^0$ is the starting point, $\bar x=P_{U\cap V}(x^0)$ is the only point in the intersection, $x_C^1 \coloneqq C_T(x^0)$ and $x_{DR}^1 \coloneqq T(x^0)$. 
 Note however, that this does not mean that C--DRM will not work for the general case. Even though $x_{DR}^1$  is closer to  $\bar x$ than $x_C^1$ is, C--DRM performed way less iterations (37) than DRM (971) to find the solution.

Below we present two examples, in Figures~\ref{fig:twoballs2} and \ref{fig:twoballs3}, where the two balls intersection is still compact but has infinitely many points. Depending on the starting point, we  achieve different results, regarding the number of iterations that both C--DRM and DRM take to converge. Moreover, Figure~\ref{fig:twoballs2} shows that the accumulation point of the sequences generated by C--DRM and DRM are not necessarily the same, with comparable iteration complexity though. In Figure~\ref{fig:twoballs3}, DRM converged after 6 iterations while C--DRM took 7 iterations.

\begin{figure}[!ht]
\begin{subfigure}[t]{0.24\textwidth}
\centering
\frame{\includegraphics[height=.2\textheight]{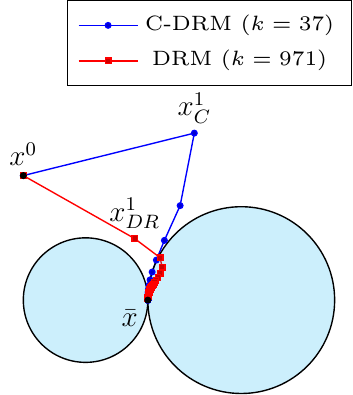}}
	 \caption{}
	\label{fig:twoballs1}
	\end{subfigure}
  \hspace*{.3in}
\begin{subfigure}[t]{0.24\textwidth}
\centering
\frame{\includegraphics[height=.2\textheight]{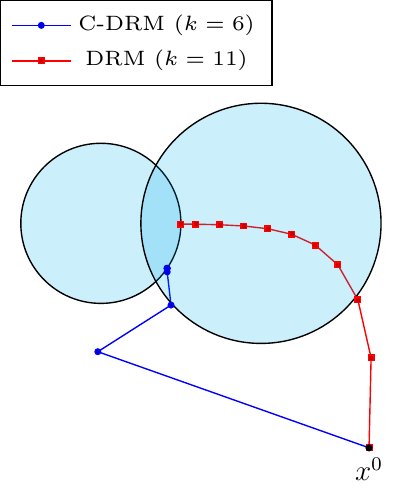}}
	 \caption{}
	\label{fig:twoballs2}
    \end{subfigure}
     \hspace*{.2in}
\begin{subfigure}[t]{0.24\textwidth}
\centering
\frame{\includegraphics[height=.2\textheight]{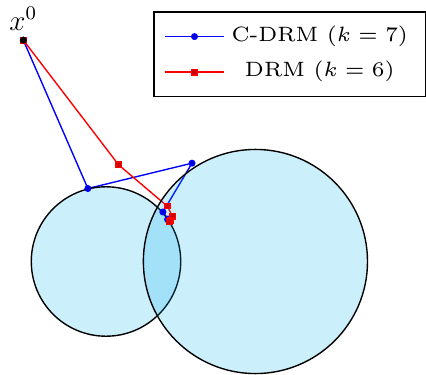}}
	 \caption{}
	\label{fig:twoballs3}

\end{subfigure}
\caption{Experiments with two balls in $\RR^2$. \label{fig:twoballs}}
\end{figure}
\end{example}

We performed extensive experiments featuring two convex balls in $\RR^n$ with starting points being chosen randomly, and the results were very similar to the ones presented in the pictures above. These positive experiments show that it might be possible to use C--DRM to find a point in the intersection of non-affine convex sets. This should be sought in the future. Note however, that C--DRM need not be defined in the general convex setting. This may be the case when an iterate happens to reach the line passing through the centers of the balls in our examples. Therefore, a hybrid strategy is necessary in order to have C--DRM generating an infinite sequence. 
\begin{example}[A ball and a line and a circle and a line in $\RR^2$]
\end{example}
Our examples show that C--DRM is likely to converge in less iterations than DRM. Also, as important as algorithmic complexity, is convergence to a solution itself. 
\begin{figure}[!ht]
\begin{subfigure}[t]{0.24\textwidth}
\centering
\frame{\includegraphics[height=.177\textheight]{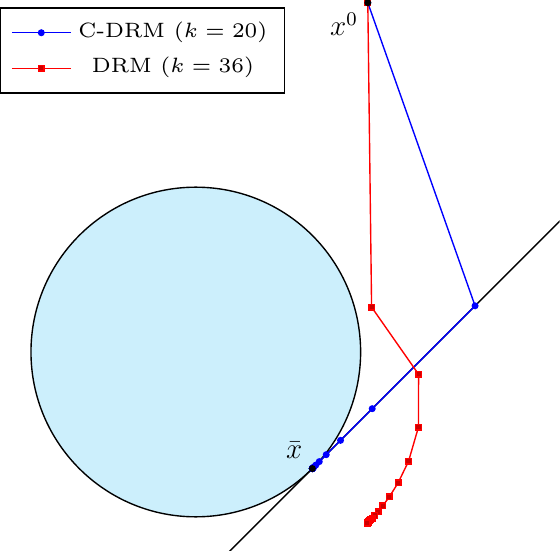}}
	 \caption{}
	\label{fig:ball-line1}
	
\end{subfigure}
   \hspace*{.21in}  
\begin{subfigure}[t]{0.24\textwidth}
\centering
\frame{\includegraphics[height=.177\textheight]{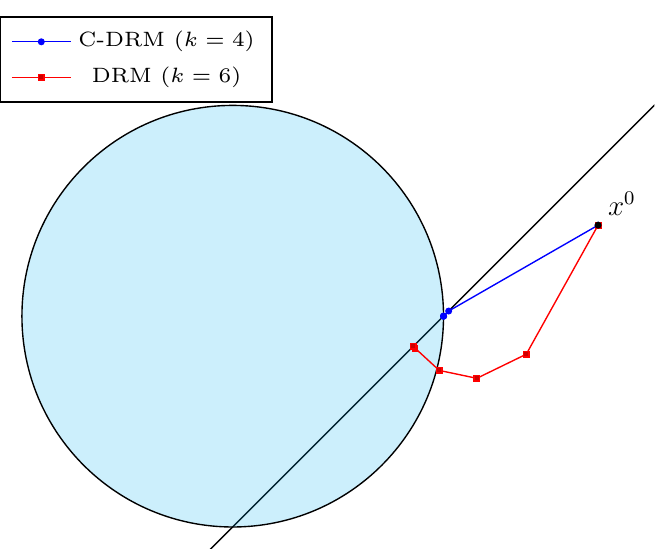}}
	 \caption{}
	\label{fig:ball-line2}
    \end{subfigure}  
    \hspace*{.45in}  
\begin{subfigure}[t]{0.2\textwidth}
\centering
\frame{\includegraphics[height=.177\textheight]{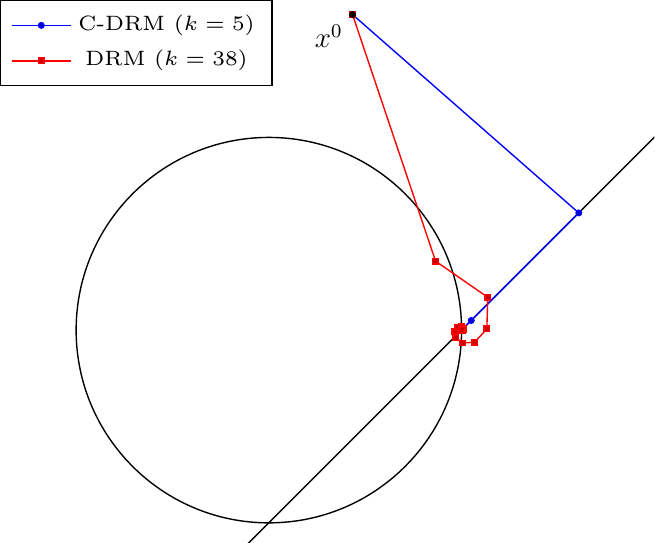}}
	 \caption{}
	\label{fig:ball-line3}
\end{subfigure}
\caption{Experiments with a ball and a line in $\RR^2$. \label{fig:ball-line}}
\end{figure}
In this regard, we underline that our pictures clearly display that C--DRM converges faster than DRM  to a solution of problem \eqref{eq:general}. 
Moreover,  DRM fails to find the unique common point of the ball and the line (only converging in the matter of shadow sequences) and to find the best approximation point in Figure~\ref{fig:ball-line1} and Figure~\ref{fig:ball-line2}, respectively.

Figure~\ref{fig:ball-line3} represents a slightly non-convex example with a circle and a line, which was considered for DRM in \cite{benoist}. In this experiment both  C--DRM and DRM converge and the latter was slower. This reveals an interesting and promising property of C--DRM, as well as DRM, when it is applied to non-convex problems. To end this discussion, it is worth mentioning that we performed extensive numerical tests for particular instances of non-convex problem in $\RR^n$ and the results are very positive. This is a humble attempt  in targeting the problem of finding a point in the intersection of   an affine subspace with the $s$--space vectors defined by the generalized $\ell_0$--ball (see problem (6) of \cite{Hesse:2014gi}).

\section{Conclusions and future work}\label{Fremarks}

We have introduced and derived a convergence analysis for the Circumcentered--Douglas--Rachford method (C--DRM) for best approximation problems featuring two (affine) subspaces $U,V\subset\RR^n$. For any initial point, linear convergence to the best solution was shown and the rate of convergence of C--DRM was proven to be at least as good as DRM's. DRM is known to converge with the sharp rate $c_F\in[0,1)$, the cosine of the Friedrichs angle between $U$ and $V$ (see, {\em e.g.},  \cite{BCNPW14}). A question, left as open in this paper, is whether $c_F$ is a sharp rate for C--DRM. In this regard, our numerical experiments show that \emph{circumcentering} ``speeds up'' convergence of DRM for two subspaces as well as for most of our simple non-affine examples.  

In view of future work, we end the paper with a brief discussion on the many set case and on how one can ``circumcenter'' other classical projection/reflection type methods.

\medskip

\noindent{\bf The many set case.} Another relevant advantage of the circumcentering scheme C--DRM is that it can be  extended to the following many set case. 

Assume that $\{U_i\}_{i=1}^{m}\subset\RR^n$ is a family of finitely many affine subspaces with nonempty intersection $\cap_{i=1}^m U_i$ and that we are interested in projecting onto $\cap_{i=1}^m U_i$  using knowledge provided by reflections onto each $U_i$. 

For an arbitrary initial point $x\in\RR^n$ we could consider a generalized C--DRM iteration by taking the circumcenter $C_T(x)$ of \[\{x,R_{U_1}(x),R_{U_2}R_{U_1}(x),\ldots ,R_{U_m}\cdots R_{U_2}R_{U_1}(x)\},\] {\em i.e.}, $C_T(x)$ is in $\aff\{x,R_{U_1}(x),R_{U_2}R_{U_1}(x),\ldots ,R_{U_m}\cdots R_{U_2}R_{U_1}(x)\}$ and $$\dist(C_T(x), x)=\dist(C_T(x), R_{U_1}(x))=\cdots=\dist(C_T(x), R_{U_m}\cdots R_{U_2}R_{U_1}(x)).$$  The fact that $C_T(x)$ is well defined and that it is precisely given by the projection of any point in $\cap_{i=1}^m U_i$ onto 
\[
  W^x \coloneqq\aff\{x,R_{U_1}(x),R_{U_2}R_{U_1}(x),\ldots ,R_{U_m}\cdots R_{U_2}R_{U_1}(x)\}
\] 
can be derived similarly as in Lemma~\ref{circum-as-proj}. We understand though, that the convergence analysis of  C--DRM for the many set case might be substantially more challenging, since we  no longer can rely on the theory of DRM. Also, one should note that now, for large $m$, the computation of $C_T(x)$ may not be negligible. This computation consists of finding the intersection of $m$ bisectors in $W^x$. Therefore, linear system~\eqref{eq:xc-linsys} is now $m\times m$, and the calculation of $C_T(x)$  requires  its  resolution.

If, for a given problem, the computation of $C_T(x)$ mentioned above is simply too demanding, one could, \emph{e.g.}, consider pairwise circumcenters or even other ways of circumcentering. The important thing here is that we can enforce several strategies to help overcome the unsatisfactory extension of the classical Douglas--Rachford method for the many set case. It is known that for an example in $\RR^2$ with three distinct lines intersecting at the original (see \cite[Example 2.1]{ABTcomb}), the natural extension of the Douglas--Rachford method may fail to find a solution. On the other hand, any reasonable circumcentering scheme will solve this particular problem in at most two steps for any initial point. Circumcentering schemes may also be embedded in methods for the many set case, {\em e.g.}, CADRA~\cite{Bauschke:2015df} and the Cyclic--Douglas--Rachford method \cite{Borwein:2014ka}.

\medskip

\noindent{\bf Circumcentering other reflection/projection type methods.}
For the case of $U$ and $V$ being affine subspaces, the reflected Cimmino method \cite{Cimmino:1938tp} considers a current point $x\in\RR^n$ and takes the next iterate as the mean $\frac{1}{2}(R_U(x)+R_V(x))$. Circumcentering the Cimmino method is possible by setting the next iterate as $\circum\{x,R_U(x),R_V(x)\}$. Something similar can be done for the Method of Alternating Projections (MAP) (see \emph{e.g.} \cite{Deutsch:1985gz}). From a point $x\in\RR^n$, MAP moves to $P_VP_U(x)$. In order to circumcenter MAP, one could take the circumcenter of $x$, $R_U(x)$ and $R_VP_U(x)$. The coherence of this approach lies on the fact that the mid points of the segments $[x,R_U(x)]$ and $[P_U(x)$, $R_VP_U(x)]$ are $P_U(x)\in U$ and $P_VP_U(x)\in V$, respectively. 
Note that both the Circumcentered--Cimmino--Method and the Circumcentered--MAP solve problems with affine subspaces $U,V\subset\RR^2$ in at most two steps. This happens because circumcentering can be seen as a $2$-dimensional hyperplane search for the two set case.

Finally, we consider that investigating C--DRM for general convex feasibility problems \cite{Bauschke:2006ej} may be fruitful.

\paragraph*{Acknowledgements}
We thank the anonymous referees for their valuable suggestions.  

\printbibliography

\end{document}